\documentclass[11pt]{article} 

\usepackage[utf8]{inputenc} 

\usepackage{geometry} 
\usepackage{graphicx} 
\usepackage{amsmath}
\usepackage{amssymb}
\usepackage{amsthm}

\usepackage{booktabs} 
\usepackage{array} 
\usepackage{paralist} 
\usepackage{verbatim} 
\usepackage{subfig} 
\usepackage{fancyhdr} 
\usepackage{sectsty}
\allsectionsfont{\sffamily\mdseries\upshape} 

\usepackage[nottoc,notlof,notlot]{tocbibind} 
\usepackage[titles,subfigure]{tocloft} 

\usepackage{amsthm}

\newtheorem{theorem}{Theorem}
\newtheorem*{untheorem}{Theorem}

\newtheorem{definition}{Definition}
\newtheorem{lemma}{Lemma}
\newtheorem{proposition}{Proposition}
\newtheorem{remark}{Remark}
\newcounter{claim}[theorem]
\newcounter{case}[theorem]
\newcounter{subcase}[case]

\newenvironment{case}[1][]{\refstepcounter{case}\par\noindent\textbf{Case \thecase:}\space#1}{}
\newenvironment{subcase}[1][]{\refstepcounter{subcase}\par\underline{Subcase \thesubcase:}\space#1}{}

\title{A topological characterization of toroidally alternating knots}
\author{Seungwon Kim}
\date{}

\begin{document}
\maketitle
\begin{abstract}
We extend Howie's characterization of alternating knots to give a topological characterization of toroidally alternating knots, which were defined by Adams. 
We provide necessary and sufficient conditions for a knot to be toroidally alternating. We also give a topological characterization of almost-alternating knots which is different from Ito's recent characterization.
\end{abstract}

\section{Introduction}

Recently, Greene \cite{Greene} and Howie \cite{Howie2, Howie} independently gave a topological characterization of alternating knots, which answered a long-standing question of Ralph Fox. Below is Howie's characterization:

\begin{untheorem}\cite{Howie2, Howie}
A non-trivial knot is alternating if and only if there exists a pair of connected spanning surfaces $\Sigma$ and $\Sigma'$ in the knot exterior such that
\begin{equation}\label{Howieoriginal} 
\chi(\Sigma) + \chi(\Sigma') + \frac{1}{2} i(\partial \Sigma, \partial \Sigma') = 2.
\end{equation}
\end{untheorem}

In \cite{Adams1}, Adams defined a \em{toroidally alternating knot} \rm as a knot which has an alternating diagram on an unknotted torus in $S^3$, such that the diagram divides the torus into a disjoint union of discs; i.e., the diagram is cellularly embedded. Toroidally alternating knots include almost-alternating knots and Turaev genus one knots. Adams showed that toroidally alternating knots which are not torus knots are hyperbolic.
In this paper, we give a topological characterization of toroidally alternating knots, extending Howie's characterization of alternating knots.

Several other generalizations of alternating knots have recently been topologically characterized. In \cite{Ito}, Ito gave a topological characterization of almost-alternating knots, which were defined by Adams in \cite{Adams2}. In \cite{Howie}, Howie defined weakly generalized alternating knots and gave a topological characterization of these knots on the torus. Furthermore, In \cite{Effie}, Kalfagianni gave a characterization of adequate knots in terms of the degree of their colored Jones polynomial.

In this paper, we consider a pair of spanning surfaces satisfying an equation similar to equation (\ref{Howieoriginal}). Theorem \ref{main} shows that in this case, the knot has a ``non-trivial" alternating diagram on the torus. Non-triviality is important because every knot  has an alternating diagram on the torus boundary of its regular neighborhood. See Figure \ref{trivialalternating}.

\begin{figure}[h]
\centering
\includegraphics[scale=0.5]{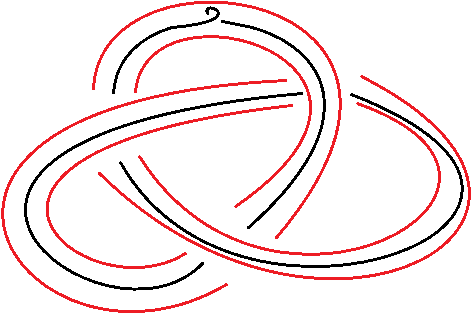}
\caption{Every knot has an alternating diagram on the torus boundary of its regular neighborhood.}
\label{trivialalternating}
\end{figure}

Theorem \ref{main} also says that if one of the spanning surfaces is free, then we can find an alternating diagram of a 	knot on an unknotted torus. When the torus is unknotted, it is a Heegaard surface, and this condition plays an important role in defining \em alternating distances, \rm which measure topologically how far a knot is from being alternating (see \cite{Lowrance} for more details.). For example, the \em alternating genus of a knot \rm is the minimal genus of a Heegaard surface such that the knot has a cellularly embedded alternating diagram on it. The \em Turaev genus \rm is another interesting alternating distance, which is the minimal genus of a Heegaard surface with a Morse function condition, such that the knot has a cellularly embedded alternating diagram on it. See \cite{ChampanerkarKofman}. Alternating genus and Turaev genus are both defined for an alternating diagram that is cellularly embedded on the surface. The conditions in Theorem \ref{main} are not enough to find a cellularly embedded diagram: The alternating diagrams on the torus that we get from Theorem \ref{main} may have an annular region and they might not be checkerboard colorable. Note that every cellularly embedded alternating diagram on a closed orientable surface is checkerboard colorable.

\begin{figure}[h]
\centering
\subfloat[An example of a cellularly embedded alternating diagram on a torus.\label{celalt}]{\includegraphics[width=0.25\textwidth]{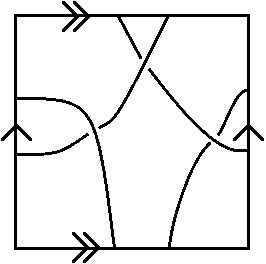}}\hfill
\subfloat[An example of a non-cellularly embedded, checkerboard-colorable alternating diagram on a torus.\label{checkalt}]{\includegraphics[width=0.25\textwidth]{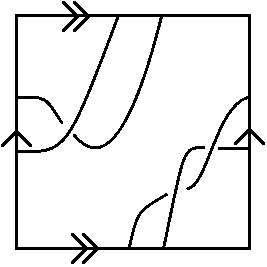}}\hfill
\subfloat[An example of an alternating diagram on a torus which is not checkerboard-colorable.\label{nocheck}]{\includegraphics[width=0.25\textwidth]{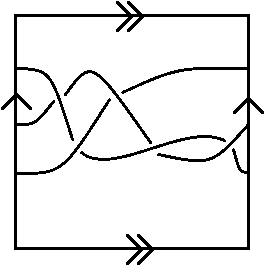}}
\end{figure}
In Theorem \ref{toroidal}, we give additional conditions -- that the spanning surfaces are {\em relatively separable}, and a liftable curve is {\em incident to a bigon} (which are defined below) -- to find a cellularly embedded alternating diagram on a torus. These conditions give a trichotomy for a pair of spanning surfaces:

\begin{enumerate}
	\item A pair of spanning surfaces is not relatively separable.
	\item A pair of spanning surfaces is relatively separable, and every liftable curve on two spanning surfaces is incident to a bigon.
	\item A pair of spanning surfaces is relatively separable, but there exists a liftable curve which is not incident to a bigon.
\end{enumerate}

Theorem \ref{toroidal} shows that a knot is toroidally alternating if and only if there exists a pair of spanning surfaces that satisfies certain conditions and either condition (1) or (2). If every pair of spanning surfaces satisfies condition (3), then we can still find some non-trivial alternating diagram on an unknotted torus by Theorem \ref{main}, but it may or may not be checkerboard colorable.

Finally, knot is \em almost-alternating \rm if it has an almost-alternating diagram that can be changed to an alternating diagram by one crossing change. In Theorem \ref{almostalternating}, we show that for any knot as in Theorem \ref{toroidal}, condition (2) is equivalent to a knot being almost-alternating. In \cite{Ito}, Ito gave a topological characterization of almost-alternating knots, but our characterization is different. He used all-$A$ and all-$B$ state surfaces of an almost-alternating diagram, which are the checkerboard surfaces of the Turaev surface of the almost-alternating diagram. We use a different pair of spanning surfaces to obtain a checkerboard-colorable alternating diagram on an unknotted torus, which is not cellularly embedded. It is an interesting question how the two checkerboard surfaces of this diagram are related to the spanning surfaces used in \cite{Ito}.

\subsection{\bf{Acknowledgements}}
The author thanks Ilya Kofman for helpful comments and guidance.

\section{Main results}
Throughout this paper, we use the following proposition that every alternating knot is both almost-alternating and toroidally alternating.

\begin{proposition}\label{alttor}
Let $K$ be an alternating knot. Then $K$ has an almost-alternating diagram and a toroidally alternating diagram.
\end{proposition}

\begin{proof}
By \cite{Adams2}, every alternating knot has an almost-alternating diagram. By \cite{Adams1}, we can find a toroidally alternating diagram from an almost-alternating diagram.
\end{proof}

\begin{definition}
A \textbf{spanning surface $\bar{\Sigma}$ of a knot $K$ in $S^3$} is a surface embedded in $S^3$ such that $\partial \bar{\Sigma} = K$.    
For $\bar{\Sigma}$, we define a \textbf{spanning surface $\Sigma$ in a knot exterior $E(K) = S^3 - int(N(K))$} by $\Sigma = \bar{\Sigma} \cap E(K)$. A spanning surface $\bar{\Sigma}$ of a knot $K$ in $S^3$ is \textbf{free} if $\pi_{1} (S^3 - \bar{\Sigma})$ is a free group.
\end{definition}  

Note that a spanning surface is free if and only if the closure of $S^3- \bar{\Sigma}$ is a handlebody.
\\

For every pair of spanning surfaces $\Sigma$ and $\Sigma'$, we can isotope them so that their boundaries realize the minimal intersection number, and each such isotopy can be extended to an isotopy of $S^3$. Then we have the following lemma from \cite{Howie2, Howie}.
\begin{lemma}[\cite{Howie2, Howie}]
If two spanning surfaces in a knot exterior are isotoped so that their boundaries realize the minimal intersection number, then every intersection arc is standard, as shown in Figure \ref{standard}.
\end{lemma}

\begin{figure}[h]
\centering
\includegraphics[scale=0.5]{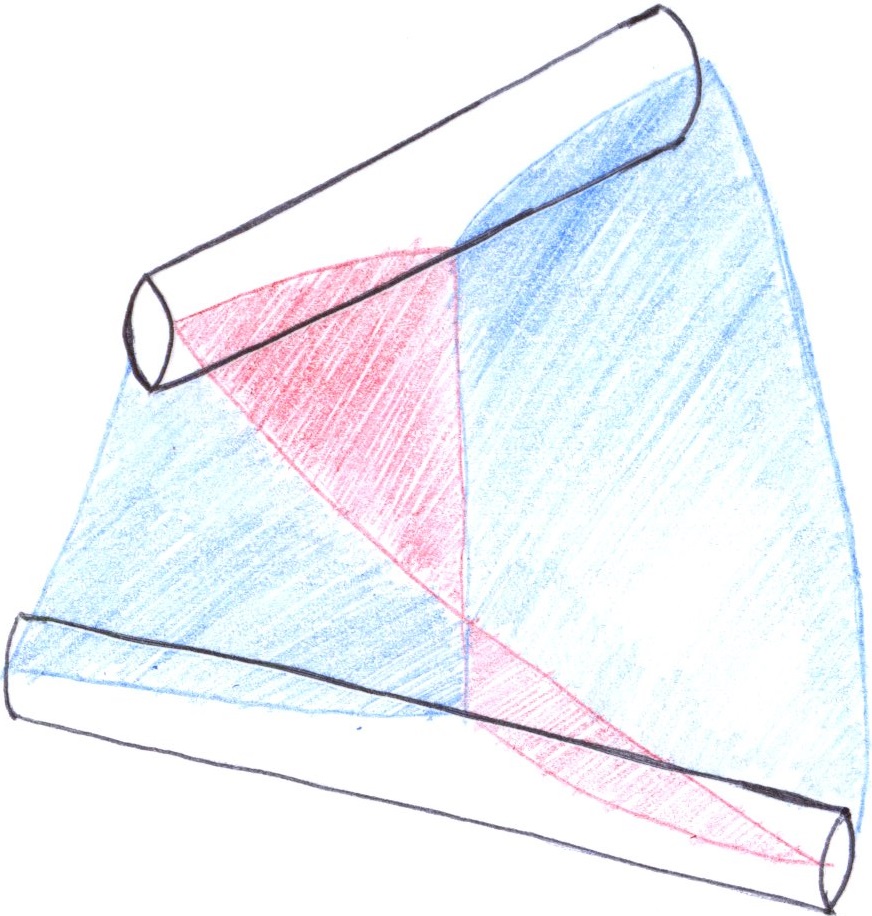}
\caption{A neighborhood of a standard intersection arc of two spanning surfaces in a knot exterior.}
\label{standard}
\end{figure}

Consider $\bar{\Sigma} \cup \bar{\Sigma'}$ in $S^3$ as above. If we contract all standard arcs, as in Figure \ref{contraction}, then we get an immersed surface in $S^3$ such that every self-intersection is a simple closed curve. We also get a connected $4$-valent graph $G_K$ on this immersed surface, coming from $K$, which is away from every self-intersection loop. We will call this immersed surface an \textbf{almost-projection surface} of $\Sigma$ and $\Sigma'$.
\begin{figure}[h]
\centering
\includegraphics[scale=0.5]{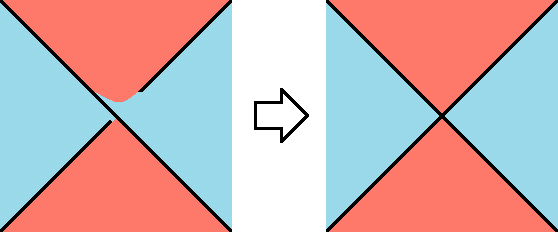}
\caption{Contracting a standard arc intersection}
\label{contraction}
\end{figure}

Lastly, we define the complexity of an alternating diagram $D_K$ of a knot $K$ on an embedded surface. Below, let $G_K$ denote the $4$-valent graph obtained from projecting $D_K$ on the surface naturally.

\begin{definition}
Let $K$ be a knot which has a diagram $D_K$ on some embedded surface $T$. 
\begin{equation}
\begin{split}
r(D_K, T) = min\{|\gamma \cap G_K| ~ \big | & ~ \gamma ~\text{is a boundary of a compressing disc of $T$} \\ 
& ~ \text{which intersects $G_K$ transversely only in edges of $G_K$.}\}
\end{split}
\end{equation} 
\end{definition}
 
\begin{theorem}\label{main}

Let $\Sigma$ and $\Sigma'$ be connected, spanning surfaces in the knot exterior $E(K)$, such that
\begin{equation} 
\label{howietype}
\chi(\Sigma) + \chi(\Sigma') + \frac{1}{2} i(\partial \Sigma, \partial \Sigma') = 0,
\end{equation}
where $i(\partial \Sigma, \partial \Sigma')$ is the minimal intersection number of $\partial \Sigma$ and $\partial \Sigma'$.
Then there exists a torus $T$ embedded in $S^3$ such that $K$ has an alternating diagram $D_K$ on $T$ with $r(D_K, T) \geq 2$.
Furthermore, if $\Sigma$ or $\Sigma'$ is free, then $T$ is an unknotted torus (i.e., a Heegaard torus).
\end{theorem}

\begin{remark}
\rm
Howie \cite{Howie} considered an alternating diagram on the torus which is checkerboard colorable. To get his characterization of weakly generalized alternating knots, he added certain other conditions. In Theorem \ref{main}, we show that without additional conditions, we can still find a non-trivial alternating diagram of the knot on the torus.
\end{remark}
\begin{proof}
First, we will prove the existence of $D_K$ and $T$. Consider an almost-projection surface of $\Sigma$ and $\Sigma'$. By equation (\ref{howietype}), this almost-projection surface can be either an immersed torus or an immersed Klein bottle.
\\
\begin{case}\label{t1ac1}
The almost-projection surface is an immersed torus.
\end{case}
\\

For each self-intersection curve, we have two possibilities. First, a self-intersection curve can bound a disc on the immersed torus. Then we can find an innermost self-intersection curve inside the disc. If we surger along the disc bounded by the innermost self-intersection curve as in Figure \ref{surgery}, then the resulting surface can be disconnected, or it is an immersed sphere. 

\begin{figure}[h]
\centering
\includegraphics[scale=0.5]{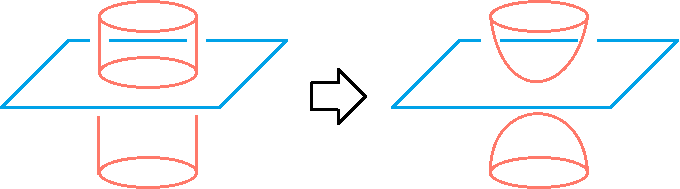}
\caption{Sugery along a disc}
\label{surgery}
\end{figure}

If the resulting surface is an immersed sphere, then by \cite{Howie2, Howie}, $K$ is alternating. Hence, by Proposition \ref{alttor}, $K$ is toroidally alternating. Also in this case, by Corollary 4.6 in \cite{Kim}, $r(D_K, T) =2$. If the resulting surface is disconnected, then one component is a torus, and the other component is a sphere. If $G_K$ is on the sphere component, then again, $K$ is alternating. Otherwise, we have reduced the number of self-intersections. We continue until all such inessential self-intersections are eliminated.

On the other hand, the self-intersection curve can be essential on the immersed torus. But we now prove by contradiction that this cannot occur. Let $f : T' \rightarrow S^3$ be an immersion map, and let $\sigma$ be a self-intersection loop on $f(T')$. Then $f^{-1} (\sigma)$ consists of two essential simple closed curves, so they bound an annulus on the torus as in Figure \ref{torus}. Since $G_K$ does not intersect with any self-intersection curves, $f^{-1}(G_K)$ is a connected 4-valent graph on one of the annuli bounded by $f^{-1}(\sigma)$. Hence, both components of $f^{-1}(\sigma)$ are in the same region of $f^{-1}(G_K)$, which implies that $\sigma$ is a self-intersection of either $\bar{\Sigma}$ or $\bar{\Sigma'}$, which contradicts the fact that each of them is an embedded surface. Hence, there are no essential self-intersection curves. 

\begin{figure}[h]
\centering
\includegraphics[scale=0.5]{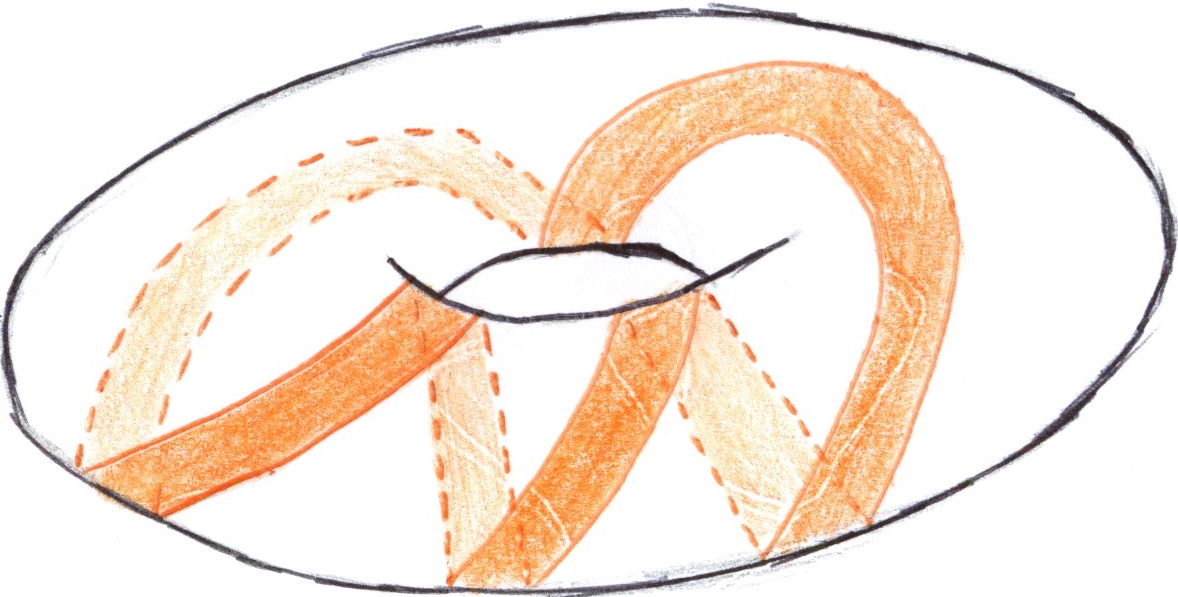}
\caption{Preimages of a self-intersection loop bound an annulus on the torus.}
 \label{torus}
\end{figure}

Therefore, $G_K$ is a $4$-valent graph on an embedded torus $T$ in $S^3$. We can recover the diagram $D_K$ from $G_K$ by replacing each vertex of $G_K$ with a neighborhood of a standard arc. If the resulting diagram $D_K$ is not alternating, there exists a bigon between $\partial \Sigma$ and $\partial \Sigma'$, which contradicts the minimality of the intersection number of boundaries. Hence, $K$ has an alternating diagram on $T$. Also, from the construction, $D_K$ is checkerboard colorable. 

We claim that either $r(D_K, T) \geq 2$ or $K$ is alternating. Suppose that for the resulting alternating diagram $D_K$ on the torus $T$, $r(D_K, T) < 2$. Since $D_K$ is checkerboard colorable, every simple closed curve on $T$ intersects $G_K$ transversely in an even number of points. Therefore, $r(D_K, T) = 0$, so we can find a compressing disc of $T$ which does not intersect $D_K$. Then compressing $T$ along this disc yields an embedded $S^2$, so $K$ is an alternating knot. Then, as above, by \cite[Corollary 4.6]{Kim}, we can find an alternating diagram $D_K$ of $K$ on some embedded torus $T''$ such that $r(D_K, T'') = 2$. This concludes the proof of Case \ref{t1ac1}.
\\


\begin{case}\label{t1ac2}
The almost-projection surface is an immersed Klein bottle.
\end{case}
\\

Let $f : B \rightarrow S^3$ be an immersion of a Klein bottle $B$. If a simple closed self-intersection curve of $f(B)$ bounds a disc on $f(B)$, then we can surger along this curve as in Figure \ref{surgery} to reduce all such inessential intersections. If $G_K$ is on an immersed sphere, then $K$ is alternating. Hence, we can assume that all preimages of the remaining self-intersections are essential simple loops of $f(B)$. Let $s_1$, $s_2$ $\subset B$ be the preimages of an essential self-intersection $\sigma$ of $f(B)$. For $i = 1,2$, we call a regular neighborhood of $s_i$ $2$-sided if it is homeomorphic to an annulus, or $1$-sided if it is homeomorphic to a M\"obius band. Furthermore, the two regular neighborhoods of $s_1$ and $s_2$ in $B$ are homeomorphic, because an annulus and M\"obius band embedded in $S^3$ cannot intersect only in the core loop. Then we have three subcases to consider, depending on the topology of $s_1$ on $B$:
\\
\begin{figure}[h]
\centering
\subfloat[An example of a $2$-sided non-separating curve on the Klein bottle.\label{2sidenon}]{\includegraphics[width=0.25\textwidth]{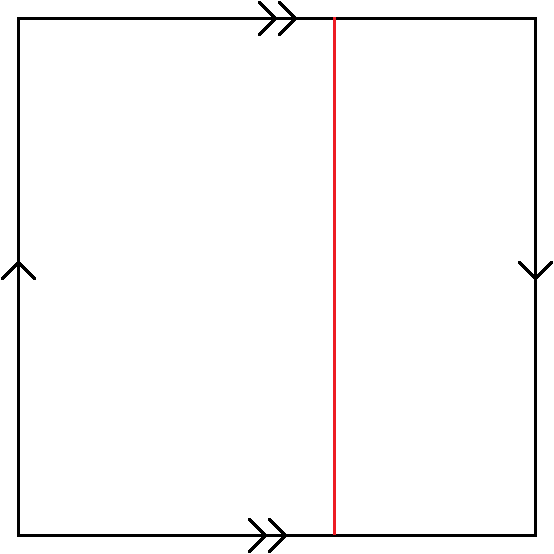}}\hfill
\subfloat[An example of a $2$-sided separating curve on the Klein bottle.\label{2side}]{\includegraphics[width=0.25\textwidth]{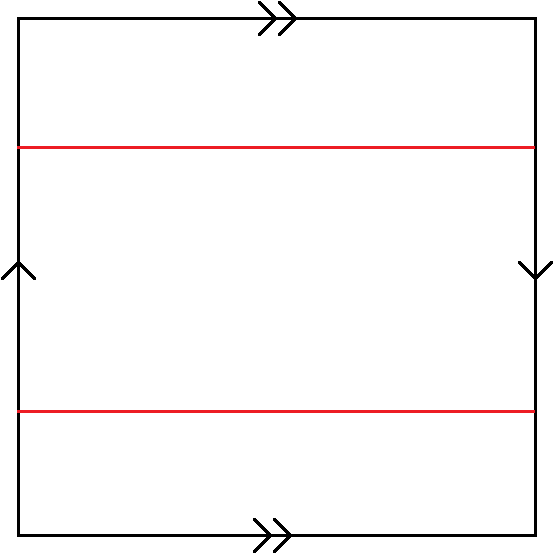}}\hfill
\subfloat[An example of a $1$-sided curve on the Klein bottle.\label{1side}]{\includegraphics[width=0.25\textwidth]{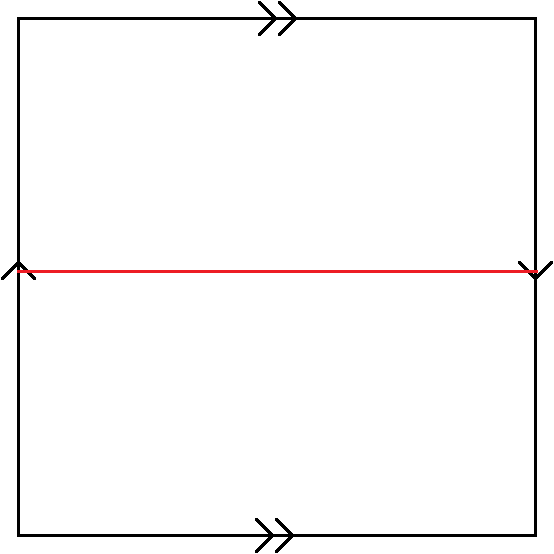}}
\end{figure}

\begin{subcase}
$s_1$ is a non-separating, $2$-sided curve on $B$. 
\end{subcase}
\\

We prove by contradiction that this subcase cannot occur. The complement of a regular neighborhood of $s_1$ in $B$ is an annulus. Hence, $s_2$ is the core of the annulus. Then $s_1$ and $s_2$ cut $B$ into two annuli, and $f^{-1}(G_K)$ is on one of them. Hence, $s_1$ and $s_2$ are in the same region of $f^{-1}(G_K)$ on $B$. See Figure \ref{Kleinbottle}. This implies that $\sigma$ is a self-intersection of $\Sigma$ or $\Sigma'$, which contradicts the assumption that $\Sigma$ and $\Sigma'$ are embedded.
\begin{figure}[h]
\centering
\includegraphics[scale=0.5]{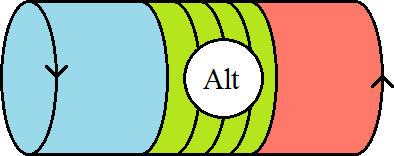}
\caption{Two pre-images cut a Klein bottle into two annuli and $f^{-1}(G_K)$ is on one of them.} 
\label{Kleinbottle}
\end{figure}
\\

\begin{subcase}
$s_1$ is a separating, $2$-sided curve on $B$.
\end{subcase}
\\

In this case, $s_1$ cuts $B$ into two M\"obius bands. Then $s_2$ is on one of the M\"obius bands, and cuts it into one M\"obius band and one annulus. Hence, $s_1$ and $s_2$ cut $B$ into two M\"obius bands and one annulus. 

\begin{figure}[h]
\centering
\includegraphics[scale=0.2]{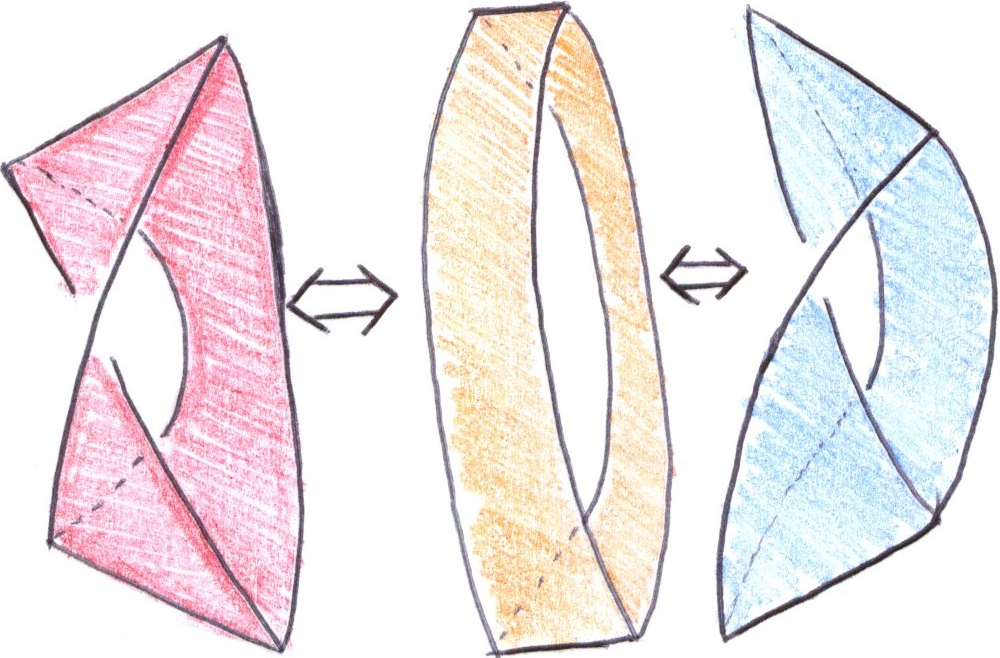}
\caption{Two pre-images cut a Klein bottle into two M\"obius bands and one annulus.}
\label{kleinbottle2} 
\end{figure}

Furthermore, $f^{-1}(G_K)$ is contained in one of the components. If $f^{-1}(G_K)$ is on the M\"obius band, then $\sigma$ is a self-intersection of $\bar{\Sigma}$ or $\bar{\Sigma'}$ which is impossible. Hence, $f^{-1}(G_K)$ is on the annulus. In this annulus, every preimage of an essential self-intersection is isotopic to a core of the annulus. Let $A$ be the annulus which contains $f^{-1}(G_K)$ and does not contain any preimages of self-intersections. Then we can recover $D_K$ from $G_K$, as in the torus case, so that $K$ is alternating on $f(A)$. 

Now, to construct the torus, we consider $B - A$, which consists of two disjoint M\"obius bands $M$ and $M'$. The image of each M\"obius band under $f$ is a subset of either $\Sigma$ or $\Sigma'$. Furthermore, both M\"obius bands cannot be contained in the same spanning surface. Consider $M \cup A$, which is homeomorphic to a M\"obius band. Now, $f(M \cup A)$ is embedded in $S^3$ because every self-intersection of $f(B)$ is an intersection of $f(M)$ and $f(M')$. Consider a thickening of the M\"obius band $f(M \cup A)$ in $S^3$, which is homeomorphic to a solid torus. Let $T$ be its boundary. Then using the natural projection, we can think of $T - \partial f(M \cup A)$ as a two fold cover of the M\"obius band $f(M \cup A)$. Then $f(A)$ is lifted to two annuli with disjoint interiors on the torus. Since $K$ is alternating on $f(A)$, we can choose a lift of the alternating diagram $D_K$ to one of the annuli. Hence, $D_K$ is alternating on $T$.

\begin{figure}[h]
\centering
\includegraphics[scale=0.5]{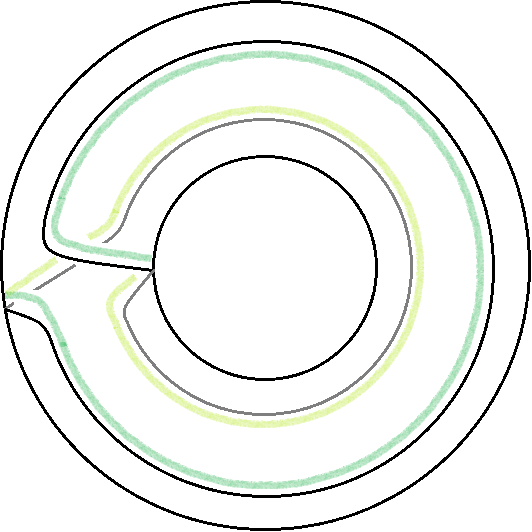}
\caption{Torus $T$, coming from thickening the M\"obius band $f(M \cup A)$, and an example of one of the lifts of $f(A)$ on $T$, denoted by a shaded band.}
\label{thickenedmobius} 
\end{figure}

Now we show that $r(D_K, T) \geq 2$ or $K$ is alternating. By construction, $T$ bounds a solid torus for which the boundary of every compressing disc intersects each lift of $f(A)$ twice. If this boundary curve intersects the diagram less than twice, then this implies that $G_K$ is contained in a disc in $f(A)$. But then, this implies that $M$ and $M'$ are in the same region of $f^{-1}(G_K)$, which cannot occur because these M\"obius bands are not contained in the same spanning surface. Hence, the boundary of every compressing disc of this solid torus intersects $G_K$ at least twice. 

Finally, $r(D_K, T) < 2$ may occur for a compressing disc on the other side of $T$.
If the $3$-manifold on the other side of $T$ has a compressing disc $\Omega$, then it is a solid torus, hence, $T$ is an unknotted torus. Note that $\partial \Omega \cap \partial f(M \cup A) \neq \emptyset$. If $\partial \Omega$ does not intersect the diagram, then just as above, $D_K$ is contained in a disc. Suppose that $\partial \Omega$ intersects the diagram once. This implies that $f(M \cup A)$ is an embedded M\"obius band in $S^3$ such that the core is the unknot and its boundary is also the unknot. This implies that we can find an essential arc on $f(A)$ which intersects $G_K$ transversely once. Hence, $K$ is a connected sum of an alternating knot and a knot which is isotopic to a boundary circle of $f(A)$, which is an unknot. Hence, $K$ is alternating.
\\
\begin{subcase}
$s_1$ is a $1$-sided curve on $B$.
\end{subcase}
\\

The complement of $s_1$ and $s_2$ in $B$ is an annulus. Hence, $f^{-1}(G_K)$ is on the annulus, and the claim follows by the same argument as in the previous subcase.  

This completes the proof of Case \ref{t1ac2}.

To show that the torus $T$ is unknotted, we need following lemmas.

\begin{lemma}\label{complexity}
Let $D_K$ be a knot diagram on the torus $T$ with $r(D_K, T) \geq 2$. Then every region of $T - G_K$ is homeomorphic to a disc, except possibly one region which is homeomorphic to an annulus.
\end{lemma}

\begin{proof}
Let $R$ be a region of $D_K$. Since $D_K$ is connected, $|T - R| = 1$. Hence, $\chi(R) \geq -1$. If $\chi(R) = -1$, then $D_K$ is contained in a disc, hence we always can find a compressing disc of $T$ which does not intersect $D_K$.  But this violates the condition $r(D_K, T) \geq 2$. Lastly, if there exist two annular regions $R_1$ and $R_2$ of $D_K$, then $|T - (R_1 \cup R_2)| = 2$.  Again, since $D_K$ is connected, this is not possible.
\end{proof}


\begin{lemma}\label{freeheegaard}
Suppose that a link $L$ has a checkerboard-colorable, connected diagram $D_L$ on a torus $T$ in $S^3$ such that $r(D_L, T) \geq 2$. Then $T$ is unknotted if and only if one of the checkerboard surfaces is free. 
\end{lemma}

\begin{proof}
From Lemma \ref{complexity}, every region is homeomorphic to a disc except possibly one region, which is homeomorphic to an annulus. Let $\Sigma$ and $\Sigma'$ be two checkerboard surfaces of $D_L$. Suppose that $\Sigma'$ is a checkerboard surface which consists with only disc regions. Since $\Sigma' - (\Sigma\cap\Sigma')$ is a set of disjoint discs and $S^3 - T$ has two connected components, $S^3 - \Sigma$ is homeomorphic to a $3$-manifold obtained from connecting the two components of $S^3 - T$ with $3$-dimensional $1$-handles, each corresponding to a disc of $\Sigma' - (\Sigma\cap\Sigma')$ (see Figure \ref{1-handle}). If $\Sigma$ contains an annular region of $D_L$, then $S^3 - \Sigma'$ can be obtained similarly, except we connect two components with a thickened annulus.  

\begin{figure}[h]
\centering
\includegraphics[scale=0.2]{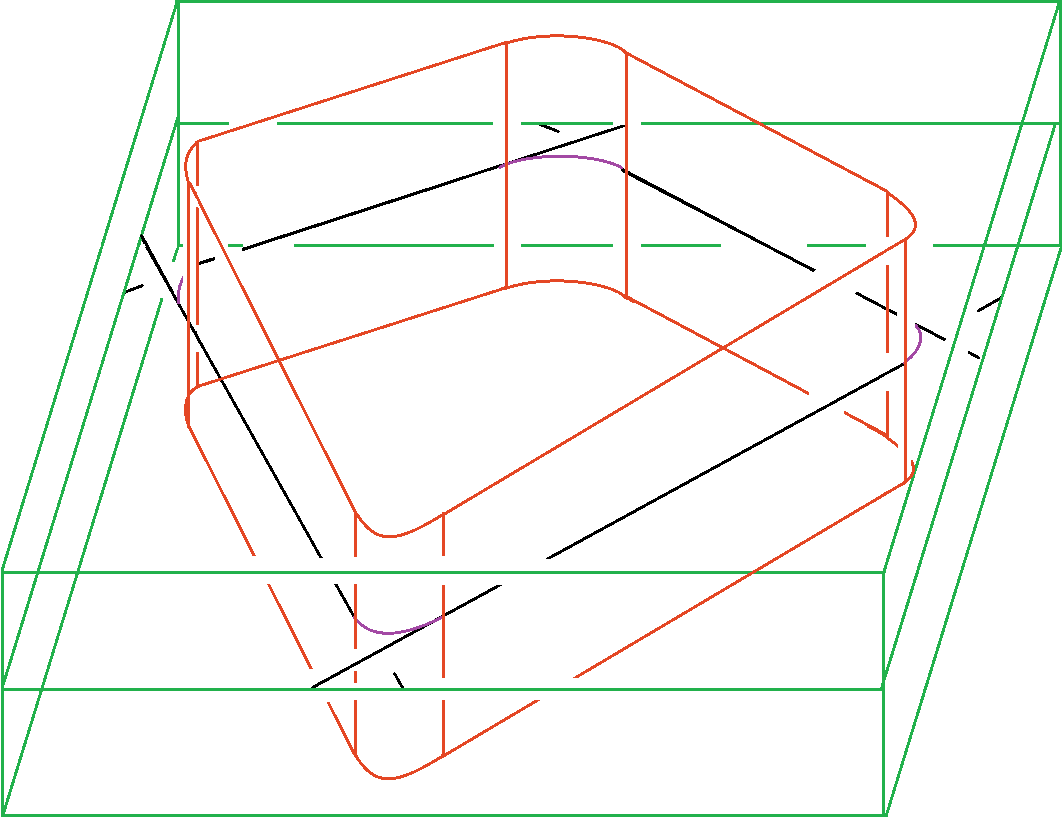}
\caption{A $1$-handle correspond to a disc region of $D_L$.}
\label{1-handle} 
\end{figure}

If $T$ is an unknotted torus, then $S^3 - T$ is a disjoint union of two solid tori. If we connect two solid tori with several $3$-dimensional $1$-handles, then it is still a handlebody. Hence, $\Sigma$ is free.

Conversely, suppose that $T$ is knotted. We show that both checkerboard surfaces are not free. We use the fact that compressing a handlebody with a disjoint set of compressing discs yields a disjoint union of handlebodies. First, we show that $\Sigma$ is not free. We can obtain $S^3 - \Sigma$ from $S^3 - T$ as above. If we compress $S^3 - \Sigma$ along all compressing discs, each corresponding to a disc of $\Sigma' - (\Sigma\cap\Sigma')$, then we get a solid torus and a $3$-manifold with boundary, which is not a solid torus, because $T$ is knotted. Hence $\Sigma$ is not free. Lastly, we show that $\Sigma'$ is not free. Consider $S^3 - \Sigma'$ and compress this manifold along all compressing discs each corresponding to a disc of $\Sigma - (\Sigma\cap\Sigma')$. Then we get a $3$-manifold which is homeomorphic to a knot exterior, such that the knot is isotopic to a core of the annular region of $D_L$. $T$ is knotted, so the core of the annular region is a non-trivial knot. So, the resulting $3$-manifold is not a handlebody. Hence, $\Sigma'$ is not free. 
\end{proof}

Now we can complete the proof of Theorem \ref{main}. Consider the almost-projection surface of $\Sigma$ and $\Sigma'$.
Suppose that the almost-projection surface is an immersed torus. If we surger the almost-projection surface along a disc, the surface might become disconnected or its genus will decrease. If the surgery reduces the genus, then we get an alternating knot by \cite{Howie2, Howie}.  As mentioned above, using the Turaev surface, such a knot is toroidally alternating. We continue performing the surgery, cutting off spheres until we get an alternating diagram $D_K$ on an embedded torus. 
As discussed above, since $D_K$ is checkerboard colorable, $r(D_K, T)$ is even. If $r(D_K, T) = 0$, then $K$ is an alternating knot. Suppose $r(D_K, T) \geq 2$. During the surgery, we cut off spheres, so the resulting checkerboard surfaces of $D_K$ on $T$ are isotopic to $\Sigma$ and $\Sigma'$. By assumption, one of them is free, hence, $T$ is unknotted by Lemma \ref{freeheegaard}.


On the other hand, suppose that the almost-projection surface is an immersed Klein bottle.
From the proof above, we can find an embedded M\"obius band, $f(M \cup A)$. Every region of $D_K$ on $f(M \cup A)$ is a disc except one annular region in $\Sigma'$ and one M\"obius band region in $\Sigma$. Consider the regular neighborhood of $\Sigma$ in $S^3$ as the following: 

We first thicken $f(M \cup A)$, and remove every thickened region of $D_K$ on $f(M \cup A)$ that is a subset of $\Sigma'$. The resulting manifold is homeomorphic to a regular neighborhood of $\Sigma$. Now, we compress the complement of $\Sigma$ by filling each thickened disc region of $\Sigma'$. Then since $\Sigma$ is free, the resulting complementary region is still a handlebody. We recover the complement of $f(M\cup A)$, so this handlebody is a solid torus. This implies that the core of the M\"obius band is unknotted in $S^3$, which implies that the solid torus that we obtained in Case \ref{t1ac2} is unknotted. 
This completes the proof of Theorem \ref{main}.
\end{proof}

\begin{definition}\cite{Howie}\label{bigon}
Let $\Sigma$ and $\Sigma'$ be properly embedded surfaces in general position in $E(K)$. A \textbf{bigon} is a disc $\mathcal{B}$ embedded in $E(K)$ such that $\partial \mathcal{B} = \beta \cup \beta'$, where $\beta \subset \Sigma$ and $\beta' \subset \Sigma'$ are connected arcs, $\beta \cap \beta'$ consists of two distinct points of $\Sigma \cup \Sigma'$ and $\mathcal{B} \cap (\Sigma \cup \Sigma') = \partial \mathcal{B}$. The arcs $\beta$ and $\beta'$ are called \textbf{edges} of $\mathcal B$, and $\beta \cap \beta'$ are called \textbf{vertices} of $\mathcal B$. A bigon is \textbf{inessential} if it can be homotoped to an intersection arc or an intersection loop of $\Sigma$ and $\Sigma'$. Otherwise, it is \textbf{essential}.    
\end{definition}
Here, the homotopy must be such that restricted to the boundary of $\mathcal{B}$, $\beta$ and $\beta'$ must remain in $\Sigma$ and $\Sigma'$, respectively, throughout the homotopy.


Let $\Sigma$ and $\Sigma'$ be a pair of spannig surfaces in $E(K)$. A {\em minimal representative} of a simple loop $\gamma$ in $\Sigma$ is a simple loop in $\Sigma$ which is isotopic to $\gamma$ and intersects $\Sigma'$ minimally. We can define a minimal representative of a simple loop in $\Sigma'$ in the same manner.  

\begin{definition}
Let $\Sigma$ and $\Sigma'$ be a pair of spanning surfaces in $E(K)$. Then $\Sigma$ and $\Sigma'$ are \textbf{relatively separable} if there exists an essential $2$-sided simple loop $\gamma$ in $\Sigma$ or $\Sigma'$ such that its push-off $\gamma'$ does not intersect the other spanning surface. Also, we call such $\gamma$ is \textbf{liftable}. In this case, $\gamma$ is \textbf{incident to a bigon} if for every minimal representative of $\gamma$, there exists an essential bigon whose boundary intersects $\gamma$ transversely in one point.
\end{definition}

\begin{figure}[h]
\centering
\includegraphics[scale=0.4]{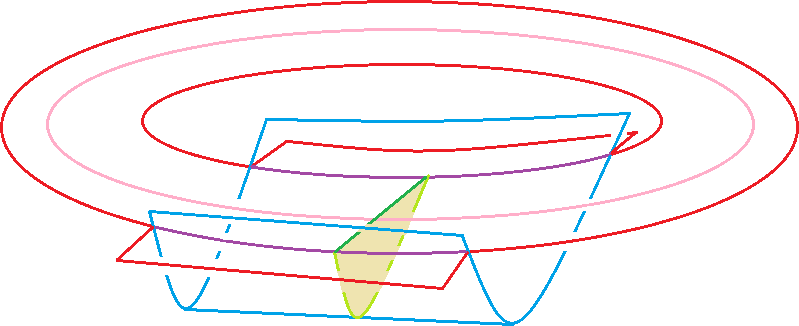}
\caption{A liftable curve which is incident to a bigon.}
\label{incident to a bigon} 
\end{figure}

\begin{definition}
Let $\Sigma$ and $\Sigma'$ be spanning surfaces in $E(K)$. We say that $\Sigma$ and $\Sigma'$ are \textbf{essentially intersecting} if their boundaries intersects minimally on $\partial E(K)$ and every intersection loop is essential on both surfaces.
\end{definition}

Since by Proposition \ref{alttor}, every alternating knot is toroidally alternating, we only consider non-alternating knots below.

\begin{theorem}\label{toroidal}
A non-alternating knot $K$ is toroidally alternating if and only if there exists a pair of essentially intersecting connected, free spanning spanning surfaces $\Sigma$ and $\Sigma'$ in the knot exterior which satisfy the following:
\begin{enumerate}
	\item $\chi(\Sigma) + \chi(\Sigma') + \frac{1}{2} i(\partial \Sigma, \partial \Sigma') = 0$.
	\item If $\Sigma$ and $\Sigma'$ are relatively separable, then every liftable curve is incident to a bigon.
	\end{enumerate}
\end{theorem}

\begin{remark}
\rm In \cite[Figure 3.18]{Howie}, Howie gave an example of a weakly generalized alternating projection of the knot $10_{139}$, for which one of the regions is homeomorphic to an annulus. He showed that there is no essential bigon between the two checkerboard surfaces $\Sigma$ and $\Sigma'$. Hence, this pair $\Sigma$ and $\Sigma'$ are an example of a pair of essentially intersecting free spanning surfaces which is relatively separable, but not every liftable curve is incident to a bigon. 
\end{remark}

\begin{proof}
First, we show that if the two checkerboard surfaces $\Sigma$ and $\Sigma'$ of a toroidally alternating diagram are relatively separable, then every liftable loop is incident to a bigon. Let $\gamma$ be a minimal representative of a liftable loop on $\Sigma$. Consider the push-off $\gamma'$ of $\gamma$. Let $\mathcal{A}$ be the annulus bounded by $\gamma$ and $\gamma'$ such that $\mathcal{A} \cap \Sigma = \gamma$. Every essential loop of $\Sigma$ intersects $\Sigma'$, so $\mathcal{A}$ intersects with $\Sigma'$. Then every intersection of $\Sigma'$ and $\mathcal{A}$ is either an arc which has its endpoints on $\gamma$ or a simple loop isotopic to $\gamma$. We can modify $\gamma'$ so that $\mathcal{A}$ only intersects $\Sigma'$ in arcs. Consider an innermost bigon $\mathcal{B}$ in $\mathcal{A}$ bounded by $\gamma$ and one of the intersection arcs. Then $\mathcal{B}$ is an essential bigon, because if $\mathcal{B}$ is inessential, we can isotope $\gamma$ and $\mathcal{A}$ to remove the intersection arc, which contradicts our hypothesis, $\gamma$ is a minimal representative. Then we can slightly isotope this bigon to intersect $\gamma$ in one point. Hence, every liftable curve is incident to a bigon. This completes the proof of the ``only if" part of Theorem \ref{toroidal}.

Now to show the ``if" part, since $\Sigma$ and $\Sigma'$ are essentially intersecting, we can contract every standard arc to get an almost-projection surface. Below, let $F$ denote the almost-projection surface of $\Sigma$ and $\Sigma'$. As we showed in the proof of Theorem \ref{main}, $F$ is either an unknotted torus or an immersed Klein bottle with no $2$-sided, non-separating self-intersection loop. 

First, suppose that $\Sigma$ and $\Sigma'$ are not relatively separable. We will show that $F$ is an unknotted torus and the alternating diagram on the almost-projection surface is cellularly embedded. 

\begin{lemma}\label{nonrelative}
Let $\Sigma$ and $\Sigma'$ be essentially intersecting spanning surfaces of a knot $K$ which are not relatively separable. Then $\Sigma$ and $\Sigma'$ are checkerboard surfaces of a cellularly embedded alternating diagram on a closed orientable surface $F$ with Euler characteristic $\chi(F) = \chi(\Sigma) + \chi(\Sigma') + \frac{1}{2} i(\partial \Sigma, \partial \Sigma')$. 
\end{lemma}

\begin{proof}
Consider the almost-projection surface $F$ of $\Sigma$ and $\Sigma'$. First, we show that $F$ cannot intersect itself in an essential simple loop. Suppose that there exists an essential simple loop intersection $\phi$. Then $\phi$ is either $1$-sided or $2$-sided. Consider one of the components $\varphi$ of the boundary of a regular neighborhood of $\phi$ on $\Sigma$. 

If $\phi$ is $2$-sided, then $\varphi$ is isotopic to $\phi$, hence essential. Furthermore, $\varphi$ has a push-off which does not intersect with $\Sigma'$. Hence, $\varphi$ is liftable, which contradicts the assumption that $\Sigma$ and $\Sigma'$ are not relatively separable.

If $\phi$ is $1$-sided, then $\varphi$ bounds a M\"obius band on $\Sigma$, which is a regular neighborhood of $\phi$ on $\Sigma$. If $\varphi$ bounds a disc on the other side, then we get a closed component, which is homeomorphic to a real projective plane. A real projective plane cannot be embedded in $S^3$, so the boundary does not bound a disc on $\Sigma$, which implies that $\varphi$ is essential. Furthermore, $\varphi$ has a push-off which does not intersect $\Sigma'$, so, it is liftable. The existence of a liftable curve contradicts the assumption that $\Sigma$ and $\Sigma'$ are not relatively separable. Hence, there cannot exist an essential simple loop intersection of $F$.

If $F$ is non-orientable, then it must have a self-intersection. However, we showed above that if $F$ has an essential self-intersection, then $\Sigma$ and $\Sigma'$ are relatively separable. Since this contradicts our hypothesis, $F$ is orientable.

We now show that the alternating diagram on $F$ is cellularly embedded. Suppose that there exists a region which is not homeomorphic to a disc. Without loss of generality, we can assume that this region is a subset of $\Sigma$. Consider the graph which is a deformation retract of this region. Any loop of this graph is an essential loop of $\Sigma$ because we can find an arc on $\Sigma$ which has both of its endpoints on $K$ and intersects this loop transversely once. Furthermore, we can find a push-off of this loop which does not intersect $\Sigma'$. Hence, it is a liftable curve, so $\Sigma$ and $\Sigma'$ are relatively separable, which contradicts our hypothesis. Therefore, the alternating diagram on $F$ is cellularly embedded. This completes the proof of Lemma \ref{nonrelative}.
\end{proof}

By Lemma \ref{nonrelative}, if $\Sigma$ and $\Sigma'$ are not relatively separable, then the almost-projection surface $F$ is an unknotted torus with a cellularly embedded alternating diagram.

Now, suppose that $\Sigma$ and $\Sigma'$ are relatively separable. If $F$ is an unknotted torus with a cellularly embedded alternating diagram, then we are done. Otherwise, by the proof of Theorem \ref{main}, we have two cases: either $F$ is an unknotted torus with a non-cellularly embedded alternating diagram, or $F$ is an immersed Klein bottle with no $2$-sided, non-separating self-intersection loop. We will show in the first case, $K$ is almost-alternating hence toroidally alternating, and that the second case is not possible. 

First, assume that $F$ is an unknotted torus with non-cellularly embedded alternating diagram. Let $\gamma$ be a core of the the annular region, which is a minimal representative of itself. We showed above that $\gamma$ is liftable. We can assume that $\gamma$ is on $\Sigma$. Now, we show that if $\gamma$ is incident to a bigon, then $K$ is almost-alternating. By assumption, there is an essential bigon $\mathcal{B}$ which intersects $\gamma$ transversely once on its boundary $\partial \mathcal{B}$. After contracting standard arc intersections to get $F$, bigon $\mathcal{B}$ becomes a disc whose interior is embedded in the complement of $F$, and $\partial \mathcal{B}$ is a loop on $F$ which intersects $G_K$ only in its vertices. The loop $\partial \mathcal{B}$ is simple, whenever both vertices of $\partial \mathcal{B}$ are on different standard arc intersections of $\Sigma$ and $\Sigma'$. If $\partial \mathcal{B}$ is simple, we can modify $\partial \mathcal{B}$ to intersect $G_K$ transversely twice on edges of $G_K$. Otherwise, $\partial \mathcal{B}$ on $F$ is a loop which has one self-intersection on some vertex of $G_K$. However, the interior of $\mathcal{B}$ does not intersect itself, so we can modify $\partial \mathcal{B}$ to be a simple loop, and intersect $G_K$ transversely twice on its edges. Since $\partial \mathcal{B}$ intersects $\gamma$ transversely once, $\mathcal {B}$ on $F$ is essential, so $\mathcal{B}$ is a compressing disc of $F$.(See Figure \ref{meridian}). This implies that $K$ can be obtained from taking $n$-full twists on two strands of some alternating knot diagram. This operation yields either an alternating knot diagram or a cycle of two alternating $2$-tangles, as defined in \cite{Kim}. By assumption, $K$ is not alternating, so it is a cycle of two alternating $2$-tangles. Then by \cite{Kim}, its Turaev genus is one, so $K$ is toroidally alternating. \\

\begin{figure}[h]
\centering
\includegraphics[scale=0.25]{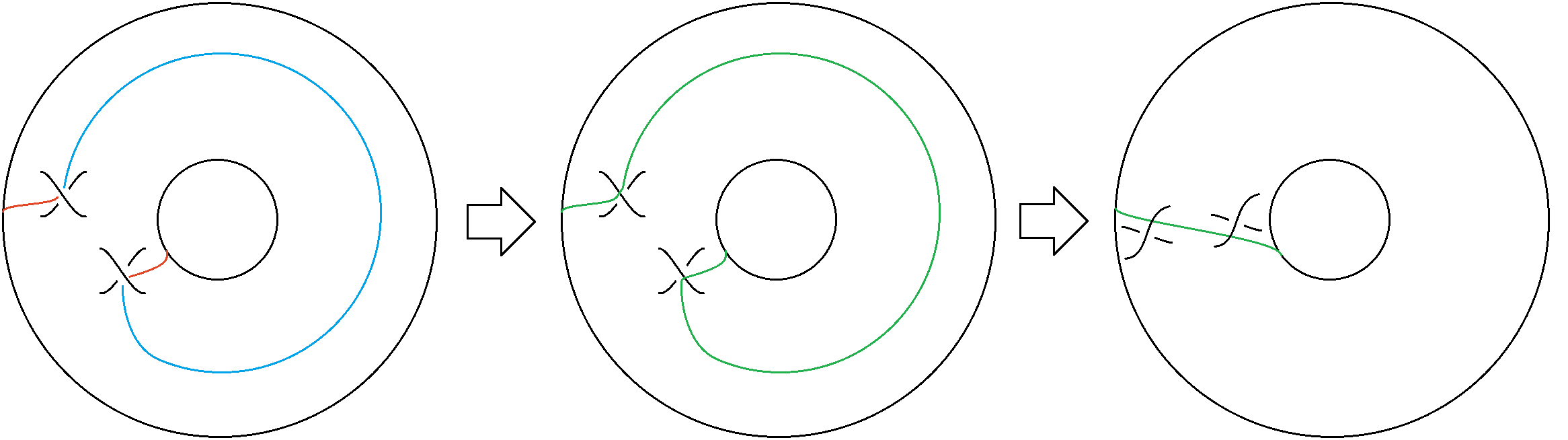}
\caption{Left : Red and blue arcs are edges of bigon. Middle : We can find an essential arc which intersects $D_K$ twice. Right : We can isotope $D_K$ so that the essential arc is on the boundary of a compressing disc.}
\label{meridian} 
\end{figure}

The remaining case is when $F$ is a Klein bottle.
We showed before that in this case, $\Sigma$ and $\Sigma'$ are relatively separable. We will show that if every liftable curve is incident to a bigon, then $K$ is alternating.


As we discussed above, every self-intersection loop of $F$ is either a $2$-sided separating loop or a $1$-sided loop. Suppose that there exists a $2$-sided self-intersection loop. Consider one of the $2$-sided self-intersection loops $\gamma'$ on $\Sigma$, which is adjacent to $G_K$ on the almost-projection surface. Consider a simple loop $\gamma$ on $\Sigma$ which is on the region between $\gamma'$ and $G_K$ on the almost-projection surface, and isotopic to $\gamma'$. We showed above that $\gamma$ is liftable. Furthermore, it is a minimal representative of itself. Consider a bigon $\mathcal{B}$ which is incident to $\gamma$. Let $\beta$ be the edge of $\mathcal{B}$ which is on $\Sigma$ and intersects $\gamma$ transversely once. Since $\gamma$ is adjacent to $G_K$, one of the vertices is on $\gamma'$ and the other is on some standard arc intersection as in Figure \ref{essentialbigon2}. This implies that there exists a properly embedded essential arc $\tau$ on the annulus which contains $G_K$, which intersects $G_K$ once.


\begin{figure}[h]
\centering
\subfloat[$\Sigma$ and $\Sigma'$ near $K$ and a liftable curve $\gamma$, which is incident to a bigon $\mathcal{B}$.\label{essentialbigon}]{\includegraphics[width=0.45\textwidth]{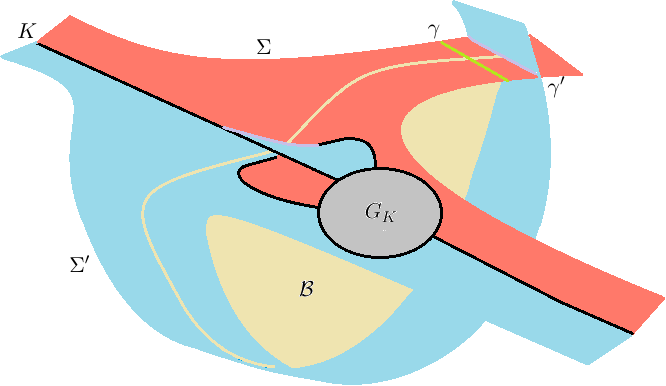}}\hfill
\subfloat[One of the vertices of $\partial \mathcal{B}$ is on $\gamma'$ and the other is on a standard arc. Then we can find $\tau$ on the annulus which contains $G_K$, which intersects $G_K$ once.\label{essentialbigon2}]{\includegraphics[width=0.45\textwidth]{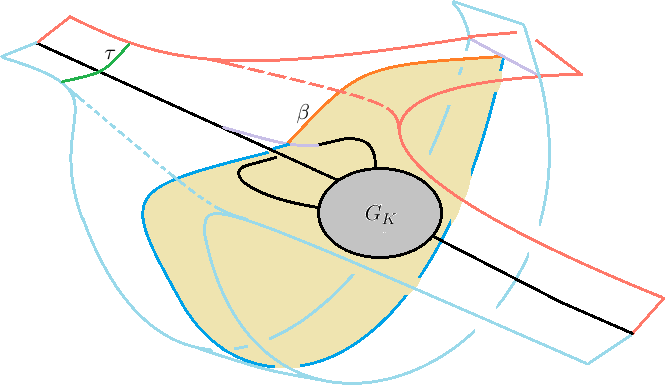}}
\caption{}
 \end{figure}

 Hence, $K$ is a connected sum of an alternating knot and the other knot, which is isotopic to a core of the annulus which contains $G_K$. In the final step of the proof of Theorem \ref{main}, we showed that when the almost-projection surface is a Klein bottle, the core of $f(M \cup A)$ is unknotted. Hence, the core of $\partial f(A)$ is a torus knot type $(2, 2q+1)$, $q \geq 0$ in $S^3$, which is alternating. Therefore, $K$ is alternating. 

Instead, suppose that there is no $2$-sided intersection loop. Then there exists a $1$-sided self-intersection loop, $\eta'$. Note that there is no other $1$-sided self-intersection loop because the complement of $\eta'$ is an annulus, so every other loop is $2$-sided. In the proof of Lemma \ref{nonrelative}, we showed that the boundary $\eta$ of a regular neighborhood of $\eta'$ on $\Sigma$ is liftable and by construction, $\eta$ is a minimal representative of itself. Consider a bigon $\mathcal{B}$ incident to $\eta$. By the same argument, $\partial \mathcal{B}$ has one vertex on $\eta'$, and the other vertex on the standard arc intersection. By the same argument as above, $K$ is alternating. 

This completes the proof of Theorem \ref{toroidal}.

\end{proof}

\begin{theorem}\label{almostalternating}
A non-alternating knot $K$ is almost-alternating if and only if there exists a pair of essentially intersecting connected, free spanning surfaces $\Sigma$ and $\Sigma'$ in the knot exterior which satisfy the following:
\begin{enumerate}
\item $\chi(\Sigma) + \chi(\Sigma') + \frac{1}{2} i(\partial \Sigma, \partial \Sigma') = 0$.  
\item $\Sigma$ and $\Sigma'$ are relatively separable and every liftable curve is incident to a bigon,
\end{enumerate}
\end{theorem}

\begin{proof}
We show the ``if" part first. In the proof of Theorem \ref{toroidal}, we showed that if $\Sigma$ and $\Sigma'$ are relatively separable and their almost-projection surface is a Klein bottle, then $K$ is alternating. Assuming that $K$ is non-alternating, the almost-projection surface of $\Sigma$ and $\Sigma'$ is an unknotted torus. Below, let $F$ denote the almost-projection surface of $\Sigma$ and $\Sigma'$. The alternating diagram $D_K$ on $F$ may be cellularly embedded or not, which we consider in separate cases.

Suppose that $D_K$ on $F$ is not cellulary embedded. Then from the proof of Theorem \ref{toroidal}, we can find an essential simple loop on $F$ which intersects $D_K$ twice and bounds a compressing disc. Hence, $K$ has a cycle of two alternating $2$-tangles, as defined in \cite{Kim}. Then by \cite[Proposition 6.6]{AbeKishimoto}, $K$ can be transformed into an almost alternating diagram. 
		
On the other hand, suppose that the diagram $D_K$ is cellularly embedded. Let $\gamma$ be a liftable curve on $\Sigma$. We will show that we can isotope $\Sigma$ and $\Sigma'$ so that after the isotopy, the new almost-projection surface of $\Sigma$ and $\Sigma'$ has an alternating diagram with an annular region whose core is $\gamma$. Let $\gamma'$ be a lift of $\gamma$, and $\mathcal{A}$ be an annulus bounded by $\gamma$ and $\gamma'$. 

Since $\gamma$ is on $\Sigma$ and $D_K$ is cellularly embedded, $\mathcal{A} \cap \Sigma' \neq \emptyset$. Furthermore, by the relatively separable condition, every intersection of $A$ and $\Sigma'$ is an arc which has both its endpoints on $\gamma$. The innermost intersection arc bounds a bigon $\mathcal{B}$. If $\mathcal{B}$ is inessential, we can isotope $\gamma$ and $\mathcal{A}$ to remove such an intersection. If $\mathcal{B}$ is essential, then we can isotope the surface along $\mathcal{B}$ to remove the intersection, as in Figure \ref{change}. After the isotopy, we get a new bigon $\mathcal{B}'$ as in Figure \ref{change}(right). 

Now, we show that after the isotopy, $\Sigma$ and $\Sigma'$ are still essentially intersecting, the new almost-projection surface of $\Sigma$ and $\Sigma'$ is still an embedded unknotted torus. Below, let $F'$ denote the new almost-projection surface of $\Sigma$ and $\Sigma'$ after the isotopy along $\mathcal{B}$.

First, we show that after the isotopy, $\Sigma$ and $\Sigma'$ are still essentially intersecting.	 Since this isotopy does not change $\Sigma$ and $\Sigma'$ near their boundaries, the number of arc intersections remains minimal. Therefore, if $\Sigma$ and $\Sigma'$ are not essentially intersecting after the isotopy, then there exists an inessential intersection loop. Furthermore, each isotopy can change the number of intersecting components at most once, so there is only one inessential intersection loop $\mu$. 

We will show that $\mu$ bounds a disc on both spanning surfaces. Since $\mu$ is inessential, it bounds a disc in one of the spanning surfaces, say $\Sigma$. If $\mu$ does not bound a disc in the other spanning surface, $\Sigma'$, then we can surger $\Sigma'$ along a disc bounded by $\mu$ on $\Sigma$. Let $\Sigma^*$ be the resulting spanning surface. Then the first condition implies that $\chi(\Sigma) + \chi(\Sigma^*) + \frac{1}{2} i(\partial \Sigma, \partial \Sigma^*) = 2$, so $K$ is alternating. As this contradicts the hypothesis that $K$ is non-alternating, $\mu$ bounds discs in both spanning surfaces.

If we undo the isotopy, then the intersection pattern of $\Sigma$ and $\Sigma'$ changes as in Figure \ref{change2}. More specifically, it changes from the right picture to the left picture. Then the edge of $\mathcal{B}$ is an arc on the left picture, which cobounds a disc with a subarc of $\Sigma \cap \Sigma'$. Then each edge can be isotoped onto $\Sigma \cap \Sigma'$, hence $\mathcal{B}$ is inessential. This contradicts the assumption that $\mathcal{B}$ is essential. Therefore, $\Sigma$ and $\Sigma'$ are essentially intersecting after the isotopy along $\mathcal{B}$.

Now, we show that $F'$ is an embedded torus. Suppose that $F'$ is a Klein bottle. A Klein bottle cannot be embedded in $S^3$ so there exists at least one self-intersection loop. Since $\Sigma$ and $\Sigma'$ intersect only in standard arcs before the isotopy, the isotopy along $\mathcal{B}$ divides one standard arc into a standard arc and an essential intersection loop (See Figure \ref{change2}). Consider a new bigon $\mathcal{B}'$ after the isotopy as in Figure \ref{change}. Then one of the vertices of $\mathcal{B}'$ is on the standard arc intersection and the other is on the essential intersection loop, so $\mathcal{B}'$ is essential. This essential simple intersection loop is a self-intersection loop of $F'$ and $F'$ is a Klein bottle, so the self-intersection loop is either $1$-sided or $2$-sided and separating. Hence, the boundary of a regular neighborhood of the essential intersection loop on each spanning surface is incident to $\mathcal{B}'$ as in Figure \ref{essentialbigon}. Then, as in the proof of Theorem \ref{toroidal}, we can show that $K$ is alternating. By hypothesis, $K$ is non-alternating, hence, so $F'$ cannot be a Klein bottle. Hence, we still have an alternating diagram on an embedded torus. 



We can continue these isotopies to remove all intersections between $\mathcal{A}$ and $\Sigma'$. Then the almost-projection surface obtained from $\Sigma$ and $\Sigma'$ after isotopy is another unknotted torus, such that the alternating diagram on the torus has an annular region whose core is $\gamma$. We will show that $\gamma$ is incident to a bigon after isotopies. Then by the same argument as in the previous case (non-cellularly embedded diagram), it follows that $K$ is almost-alternating. 

First, we will show that the new bigon $\mathcal{B}'$ after the isotopy (as in Figure \ref{change}) is an essential bigon. Suppose $\mathcal{B}'$ is inessential. Then by definition, it can be homotoped to a standard arc intersection. This implies that two vertices of $\mathcal{B}'$ are on the same standard arc intersection and both edges of $\mathcal{B}'$ are homotopic to a subarc of the standard arc intersection on each spanning surface. Let $\beta$ be an edge of $\mathcal{B}'$ on $\Sigma$. Then $\beta$ and the subarc of the standard arc intersection cobound a disc on $\Sigma$ as in Figure \ref{change2}(left). If we undo the isotopy that we performed, we have an inessential loop intersection, which contradicts the assumption that $\Sigma$ and $\Sigma'$ are essentially intersecting. Hence, $\mathcal{B}'$ is essential. 

Furthermore, by the argument similar to \cite[Proposition 2.3]{MenascoThistlethwaite}, $\partial \mathcal{B'}$ is not an inessential curve on $F'$. Hence, $\mathcal{B'}$ is a compressing disc of $F'$. Consider $\partial \mathcal{B}'$ as a meridian of $F'$.   

Lastly, we will show that $\gamma$ is incident to $\mathcal{B}'$. It is equivalent to say that $\gamma$ intersects a meridian transversely once. If $\gamma$ is an inessential simple loop on $F'$, then the diagram is disconnected, which contradicts the assumption that $K$ is a knot. If $\gamma$ is isotopic to a meridian, then we can compress $F'$ along $\gamma$ to get an alternating diagram on a sphere. Hence, $\gamma$ intersects a meridian at least once. If $\gamma$ intersects a meridian more than once, then the meridian of $F'$ intersects the diagram more than twice. This implies that $\partial \mathcal{B}'$ intersects the diagram more than twice, but we showed in the proof of Theorem \ref{toroidal} that $\partial \mathcal{B}'$ intersects the diagram transversely twice. Hence, $\gamma$ intersects a meridian transversely once, which is equivalent to say that $\gamma$ is incident to $\mathcal{B}'$. 


\begin{figure}[h]
\centering
\includegraphics[scale=0.3]{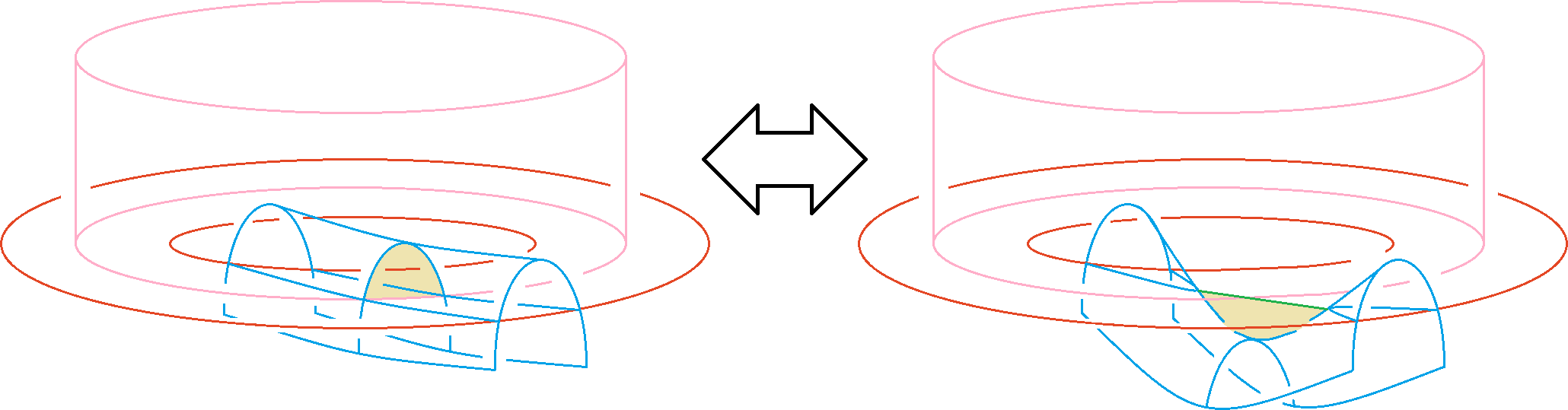}
\caption{A surface isotoped along a bigon.}
\label{change}
\end{figure}

\begin{figure}[h] 
\centering
\includegraphics[scale=0.5]{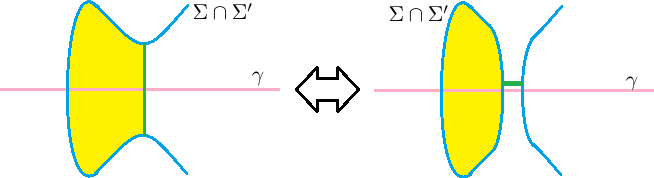}
\caption{Intersection pattern changes whenever we isotope a surface along a bigon.}
\label{change2}
\end{figure}

Now, to show the ``only if" part, suppose that the knot $K$ is almost-alternating. Consider almost-alternating diagram of $K$ as in Figure \ref{almostalternatingdiagram}(left). Then we can do a Reidemeister II move as in Figure \ref{almostalternatingdiagram}(middle) to make the diagram as in Figure \ref{almostalternatingdiagram}(right).
\begin{figure}[h]
\centering
\includegraphics[scale=0.3]{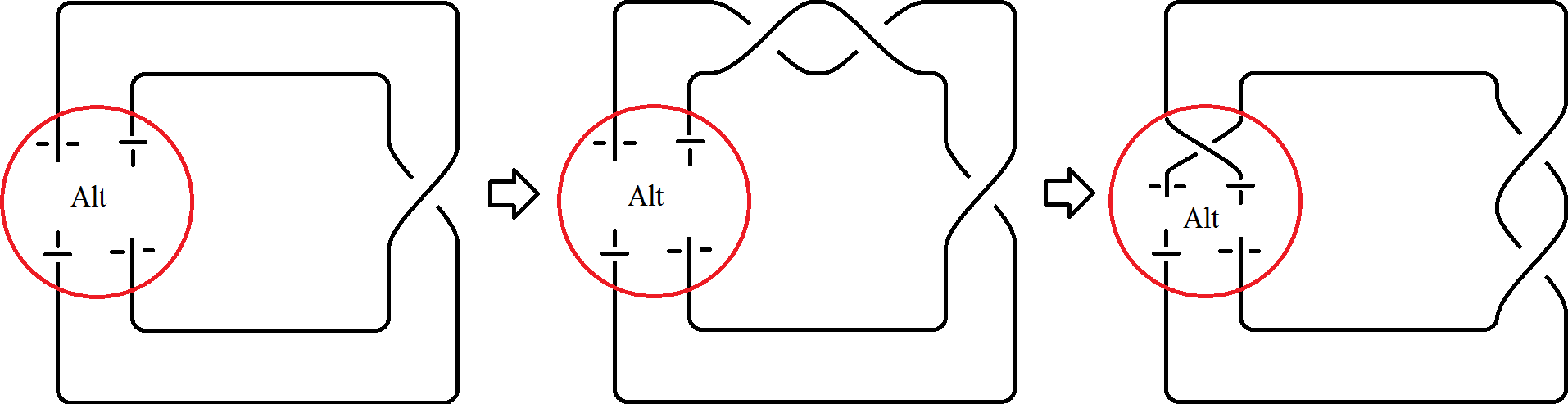}
\caption{}
\label{almostalternatingdiagram}
\end{figure}
Then $K$ has a checkerboard-colorable, non-cellularly embedded, alternating diagram $D_{K}'$ on an unknotted torus as in Figure \ref{different}(left). By Lemma \ref{complexity}, $D_{K}'$ has a unique annular region. The core of the annular region is liftable, so the two checkerboard surfaces of $D_{K}'$ are relatively separable. Now we need to show that every liftable curve is incident to a bigon. The core of the annular region is incident to the bigon shown in Figure \ref{different}(right). This bigon is essential because the two vertices of this bigon are contained in different standard arc intersections. If there exists another liftable curve, then it must intersect the standard arc intersection. Hence, as in the proof of the ``if" part of this theorem, we can find an essential bigon using the annulus bounded by the liftable curve and its push-off. 

This completes the proof of Theorem \ref{almostalternating}.

\begin{figure}[h]
\centering
\includegraphics[scale=0.25]{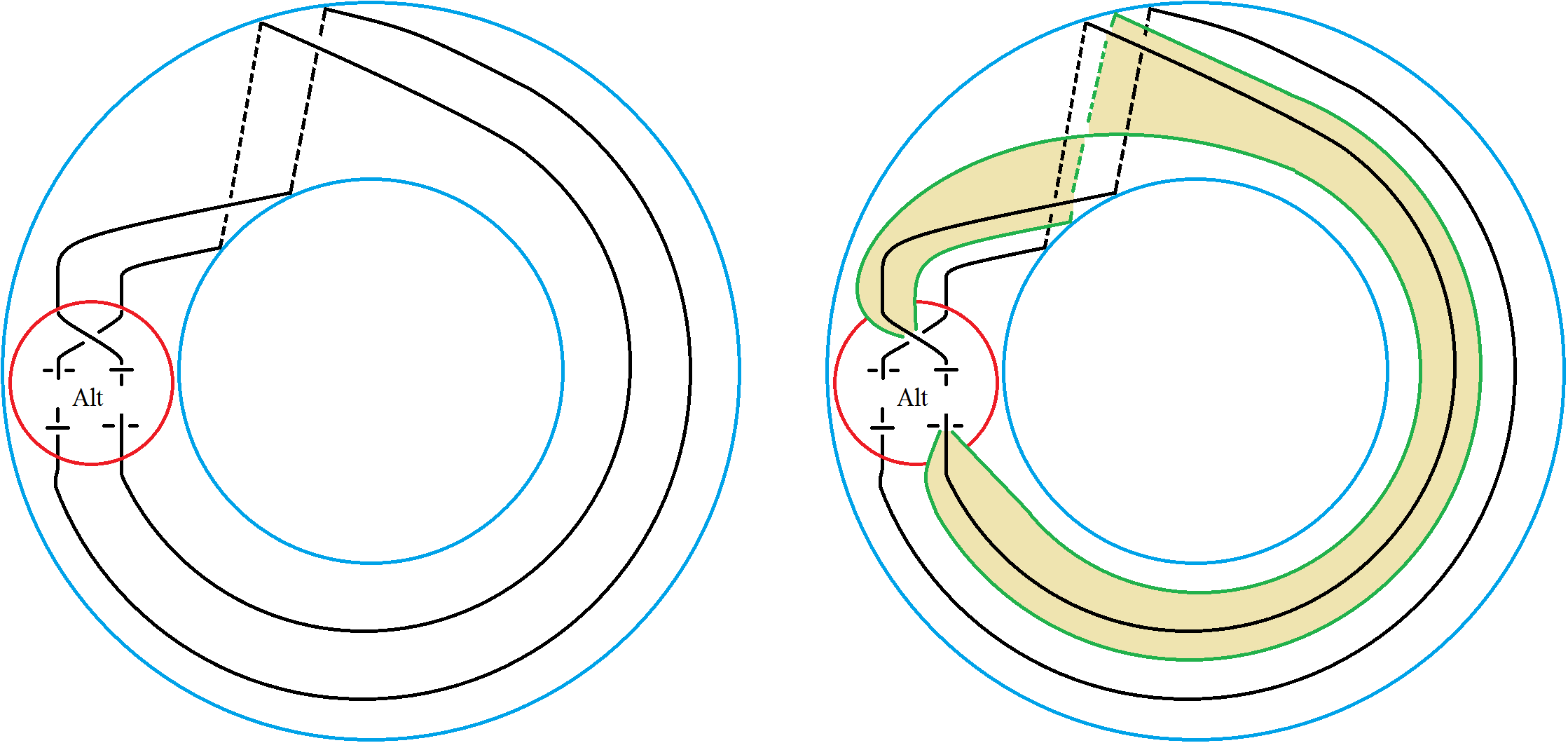}
\caption{}
\label{different}
\end{figure}

\end{proof}



\begin{thebibliography}{99}
\bibitem{Abe}
Abe, Tetsuya. 
\textit{``The Turaev genus of an adequate knot."}
Topology and its Applications 156.17 (2009): 2704-2712.

\bibitem{AbeKishimoto}
Abe, Tetsuya, and Kishimoto, Kengo . 
\textit{``The dealternating number and the alternation number of a closed 3-braid."}
Journal of Knot Theory and its Ramifications 19.09 (2010): 1157-1181.

\bibitem{Adams1}
Adams, Colin C. 
\textit{``Toroidally alternating knots and links."}
Topology 33.2 (1994): 353-369.

\bibitem{Adams2}
Adams, Colin C., et al. 
\textit{``Almost alternating links."}
Topology and its Applications 46.2 (1992): 151-165.

\bibitem{ArmondLowrance}
Armond, Cody W., and Lowrance, Adam M. 
\textit{``Turaev genus and alternating decompositions."} 
arXiv preprint arXiv:1507.02771 (2015).






\bibitem{ChampanerkarKofman}
Champanerkar, Abhijit, and Kofman, Ilya. 
\textit{``A survey on the Turaev genus of knots."}
Acta Mathematica Vietnamica 39.4 (2014): 497-514.

\bibitem{Greene}
Greene, Joshua Evan. 
\textit{``Alternating links and definite surfaces."}
arXiv preprint arXiv:1511.06329 (2015).


\bibitem{Howie2}
Howie, Joshua. 
\textit{``A characterisation of alternating knot exteriors."}
arXiv preprint arXiv:1511.04945 (2015).

\bibitem{Howie}
Howie, Joshua Andrew. 
\textit{``Surface-alternating knots and links."}
Ph.D Thesis, The University of Melbourne 2015.

\bibitem{Ito}
Ito, Tetsuya. 
\textit{``A characterization of almost alternating knot."}
 arXiv preprint arXiv:1606.00558 (2016).

\bibitem{Effie}
Kalfagianni, Efstratia. 
\textit{``A characterization of adequate knots."}
arXiv preprint arXiv:1601.03307 (2016).

\bibitem{Kim}
Kim, Seungwon. 
\textit{``Link diagrams with low Turaev genus."}
arXiv preprint arXiv:1507.02918 (2015).

\bibitem{Lowrance}
Lowrance, Adam M. 
\textit{``Alternating distances of knots and links."}
Topology and its Applications 182 (2015): 53-70.

\bibitem{MenascoThistlethwaite}
Menasco, William, and Thistlethwaite, Morwen. 
\textit{``The classification of alternating links."}
Annals of Mathematics 138.1 (1993): 113-171.

\bibitem{Thistlethwaite}
Thistlethwaite, Morwen B. 
\textit{``An upper bound for the breadth of the Jones polynomial."}
Math. Proc. Cambridge Philos. Soc 103.3 (1988): 451-456.



\end{thebibliography}
\end{document}